\documentclass[10pt,notitlepage,twoside,a4paper]{amsart}
 \usepackage{amsfonts}

\usepackage{amsmath,amssymb,enumerate}

\usepackage{epsfig,fancyhdr,color}

\usepackage{amssymb}
\usepackage{amsmath,amsthm}
\usepackage{latexsym}
\usepackage{amscd}
\usepackage{psfrag}
\usepackage{graphicx}
\usepackage[latin1]{inputenc}
\usepackage[all]{xy}
\usepackage[mathcal]{eucal}

\definecolor{NoteColor}{rgb}{1,0,0}

\renewcommand{\textsc}{\textcolor{red}}

\newtheorem{theorem}{\rm\bf Theorem}[section]
\newtheorem{proposition}[theorem]{\rm\bf Proposition}
\newtheorem{lemma}[theorem]{\rm\bf Lemma}
\newtheorem{corollary}[theorem]{\rm\bf Corollary}
\newtheorem*{theorem 1}{\rm\bf Proposition 1}
\newtheorem*{theorem 2}{\rm\bf Proposition 2}

\theoremstyle{definition}

\theoremstyle{remark}

\newtheorem{question}[theorem]{\rm\bf Question}

\def\interieur#1{\mathord{\mathop{\kern 0pt #1}\limits^\circ}}

\title[Finsler structure]
{On the Finsler stucture of the Teichm\"uller metric and Thurston's asymmetric metric}

\author{A. Papadopoulos}
\address{Athanase Papadopoulos,  Universit{\'e} de Strasbourg and CNRS,
7 rue Ren\'e Descartes,
 67084 Strasbourg Cedex, France}
\email{athanase.papadopoulos@math.unistra.fr}

\author{W.  Su}
\address{Weixu Su, Department of Mathematics, Fudan University, 200433, Shanghai, P. R. China, and Universit\'e de Strasbourg and CNRS, 7 rue Ren\'e Descartes,
67084 Strasbourg Cedex, France}
\email{suweixu@gmail.com}

\date{\today}

\begin{document}

\begin{abstract}
We highlight several analogies between the Finsler (infinitesimal) properties of Teichm\"uller's metric and Thurston's asymmetric metric on Teichm\"uller space. Thurston defined his asymmetric metric in analogy with Teichm\"ullers' metric, as a solution to an extremal problem, which consists, in the case of the asymmetric metric, of finding the best Lipschitz maps in the hoomotopy class of homeomorphisms between two hyperbolic surface. (In the Teichm\"uller metric case, one searches for the best quasiconformal map between two conformal surfaces.) It turns out also that some properties of Thurston's asymmetric metric can be used to get new insight into Teichm\"uller's metric. In this direction, in analogy with Thurston's formula for the Finsler norm of a vector for the asymmetric metric that uses the hyperbolic length function, we give a new formula for the Finsler norm of a vector for the Teichm\"uller metric that uses the extremal length function. We also describe an embedding of projective measured  foliation space in the cotangent space to Teichm\"uller space whose image is the boundary of the dual of the unit ball in the tangent space representing vectors of norm one for the Finsler structure associated to the Teichm\"uller metric.

\end{abstract}

\maketitle


\medskip

\medskip


\section{Introduction}\label{intro}
Let $S=S_{g,n}$ be an oriented surface of genus $g\geq 0$ with $n\geq 0$ punctures. We assume that $3g-3+n > 0$.
The Teichm\"uller space
$\mathcal{T}(S)$ of $S$ is the space of complex structures (or, equivalently, of complete and finite-area hyperbolic structures) on $S$
up to equivalence, where two such structures $X$ and $Y$ are considered as equivalent if there is a conformal map (respectively, an isometry) $h:(S,X) \to (S,Y)$ which is
homotopic to the identity map of $S$.

There are several natural metrics on $\mathcal{T}(S)$; some of them are Riemannian and others are only Finsler. Some metrics on Teichm\"uller space (e.g. the Weil-Petersson metric, and the Teichm\"uller metric) are strongly related to the complex structure of that space. A Finsler metric is a length metrics where the distance between two points is defined by minimizing lengths of peicewise $C^1$ paths joining them and where the length of a path is computed by integrating norms of tangent vectors. The fact that a metric is Finsler but not Riemannian means that the norm function on each tangent space is not associated to a scalar product. In this paper, we consider the Teichm\"uller metric and Thurston's asymmetric metric. Both metrics are Finsler, and in the case of Thurston's asymmetric metric, the Finsler norm not only is non-Riemannian but it is even not symmetric. These metrics are defined in terms of distances between geometric structures on the surface (conformal structures, in the case of the Techm\"uller metric, and hyperbolic structures, in the case of Thurston's asymmetric metric). The question of comparing the various metrics on $\mathcal{T}(S)$ is a natural one, and it was also suggested by Thurston in \cite{Thurston}; see also \cite{PT}.

We shall recall below the definition of Thurston's asymmetric metric.
Thurston formulated several ideas concerning this asymmetric metric, and we shall review some of them below. He proved some of the major results and he outlined some other results in his paper \cite{Thurston}. After Thurston's paper was circulated, questions on the limiting behavior of stretch lines  and of anti-stretch lines were considered, see \cite{P1}, \cite{PT}, \cite{T3} and \cite{T2}. Stretch lines are some special geodesics of that metric, and anti-stretch lines are stretch lines traversed in the opposite directions; we note that since the metric is not symmetric, a geodesic traversed in the opposite direction is not necessarily a geodesic. There were also generalizations of Thurston's asymmetric metric to the case of surfaces with boundary, see \cite{LSPT1} and \cite{LSPT2}. A renewal of interest in this metric has emerged recently, see e.g. the papers  \cite{Walsh}, \cite{LRT} and \cite{LPST}. It became also clear  that among the known metrics on Teichm\"uller space, the Thurston metric  is the one that has an interesting analogue on outer space, see \cite{CV}. Thus, there are very good reasons to study Thurston's asymmetric metric.
Some analogies and some differences between Thurston's asymmetric metric and Teichm\"uller's metric have described in the paper \cite{PT}. In the present paper, we point out new analogies that concern the Finsler character of these metrics.

Given a set $X$, a function $d:X\to X$ is said to be a \emph{weak metric} on $X$ if it satisfies all the axioms of a distance function except the symmetry axiom. The weak metric $d$ is said to be \emph{asymmetric} if it is strictly weak, that is if there exist two points $x$ and $y$ in $X$ such that $d(x,y)\not=d(y,x)$.

A related notion is that of a \emph{weak norm} on a vector space $E$. This is a function $p:E\to \mathbb{R}$  that is nonnegative, convex and positively homegeneous. In other words, $p$ satisfies the following:
\begin{enumerate}\label{weak-norm}
\item $p(x)\geq 0$ for all $x\in E$;
\item \label{wn1} $p(\lambda x)=\lambda p(x)$ for all $x\in E$ and for all $\lambda\geq 0$;
\item \label{wn2} $p(x_1+x_2) \leq p(x_1)+ p(x_2)$ for all $x_1,x_2 \in E$.
\end{enumerate}

In this paper, to simplify terminology, we shall use in general the term \emph{norm} to denote a weak norm a,d the terms \emph{metric} to denote a weak metric.

Let $X$ and $Y$ be now two complete hyperbolic metrics on $S$.
In the paper \cite{Thurston}, Thurston defined an asymmetric metric $d_L$ on  $\mathcal{T}(S)$ by setting
\begin{equation} \label{eq:L}
d_L(X,Y)=\inf_{f} \log L_f(X,Y),
\end{equation}
where the infimum is taken is over all homeomorphsims $f:X \to Y$ homotopic to the indentity map of $S$, and where $L_f(X,Y)$ is the Lipschitz constant of $f$, that is,
\begin{displaymath}\label{Lip}
\hbox{Lip}(f)=\sup_{x\neq y\in S}\frac{d_{Y}\big{(}f(x),f(y)\big{)}}{d_{X}\big{(}x,y\big{)}}.
\end{displaymath}

We shall call this weak metric Thurston's asymmetric metric or, for short, Thurston's metric.

In the same paper, Thurston proved  that  there is a (non-necessarily unique) extremal Lipschitz homeomorphsim that realizes the infimum in (\ref{eq:L}). He also proved that we have the following formula for the asymmetric metric:
\begin{equation} \label{eq:LL}d_L(X,Y)=\log \sup_{\gamma}  \frac{\ell_Y(\gamma)}{\ell_X(\gamma)},
\end{equation}
where $\ell_X(\gamma)$ denotes the hyperbolic length of $\gamma$ with respect to the metric $X$ and $\gamma$ ranges over all essential simple closed curves on  $S$.

Furthermore, Thurston proved that the asymmetric metric defined in (\ref{eq:L}) is Finsler, that is, it is a length metric which is defined by integrating a weak norm on the tangent bundle of $\mathcal{T}(S)$ along paths in Teichm\"uller space, and taking the minimum lengths over all peicewise $C^1$-paths. Thurston gave an explicit formula for the weak norm of a tangent vector $V$ at a point $X$ in $\mathcal{T}(S)$, namely,

\begin{equation}\label{eq:norm-L}
\Vert V \Vert_L= \sup_{\lambda \in \mathcal{ML}} \frac{d \ell_{\lambda}(V)}{\ell_{\lambda}(X)}.
\end{equation}
Here, $\mathcal{ML}=\mathcal{ML}(S)$ is the space of measured laminations on $S$, $\ell_{\lambda}:\mathcal{T}(S)\to \mathbb{R}$ is the length function on Teichm\"uller space associated to the measured lamination $\lambda$ and $d \ell_{\lambda}$ is the differential of the function $\ell_{\lambda}(X)$ at the point $X\in \mathcal{T}(S)$.

For $X$ in $\mathcal{T}(S)$ and $\lambda$ in $\mathcal{ML}$, we shall use the notation $\ell_\lambda(X)$ or $\ell_X(\lambda)$ to denote the $X$-length of $\lambda$, depending on whether we consider the length function as a function on Teichm\"uller space or on measured lamination space.
		
We shall present an analogue of Formula (\ref{eq:norm-L})
for the Teichm\"uller metric that is expressed in terms of extremal length, namely, we show that the Finsler weak norm associated to the Teichm\"uller distance is given by the following formula:

\begin{equation}\label{eq:norm-T}
\Vert V\Vert_E= \sup_{\lambda \in \mathcal{MF}} \frac{d \mathrm{Ext}^{1/2}_{\lambda}(V)}{ \mathrm{Ext}^{1/2}_{\lambda}(X)} .
\end{equation}
Here, $X$ is a conformal structure on $S$, considered as a point in Teichm\"uller space, $V$ is a tangent vector at the point $X$, $\mathcal{MF}=\mathcal{MF}(S)$ is the space of measured laminations on $S$,  $\mathrm{Ext} _{\lambda}:\mathcal{T} \to \mathbb{R}$ is the extremal length function associated to the measured foliation $\lambda$ and $d\mathrm{Ext} _{\lambda}$ is the differential of that function at $X$.

There is a more geometric version of (\ref{eq:norm-T}) in which the tangent vector $V$ is interpreted as a Beltrami differential (see Corollary \ref{coro:teich} below).

The following table summarizes some analogies between notions and results associated to Thurston's asymmetric metric and those associated to the
Teichm\"uller metric.
\bigskip
\begin{center}
\begin{tabular}{|c | l | l |}
\hline
{} & Thurston's metric & Teichm\"uller metric   \\
\hline
(i) & Stretch maps  & Teichm\"uller extremal maps    \\
\hline
(ii) & Stretch lines  & Teichm\"uller lines    \\
\hline
(iii) & $d_L(X,Y)=\displaystyle \log \inf_f L(f)$ & $ d_T(X,Y)= \displaystyle\log \inf_f K(f) $\\
\hline
(iv) & $d_L(X,Y)=\log \displaystyle \sup_{\gamma} \displaystyle \frac{\ell_\gamma(Y)}{\ell_\gamma(X)}$ & $d_T(X,Y)= \log \displaystyle \sup_{\gamma}\displaystyle \frac{\mathrm{Ext}^{1/2}_{\gamma}(Y)}{\mathrm{Ext}^{1/2}_{\gamma}(X)}$ \\
\hline
(v) & $\| V \|_L= \displaystyle \sup_{\lambda \in \mathcal{ML}} \displaystyle \frac{d \ell_{\lambda}(V)}{\ell_{\lambda}(X)}$ & $\| V \|_T= \displaystyle \sup_{\lambda \in\mathcal{MF}} \displaystyle \frac{d \mathrm{Ext}^{1/2}_{\lambda}(V)}{\mathrm{Ext}^{1/2}_{\lambda}(X)} $   \\
\hline

 (vi) & Thurston cataclysm coordinates  & The homeomorphism    \\
&  $X \in \mathcal{T} (S) \mapsto F_\mu(X)\in\mathcal{MF}(\mu)$ & $  X\in \mathcal{T} (S) \mapsto F_h(\Phi_F(X))\in\mathcal{MF}(F)$   \\
\hline
(vii) & $\ell_\lambda(X)=i(\lambda, F_{\lambda^*}(X)) $ & $\mathrm{Ext}_{\lambda}(X)=i(\lambda, F_h(\Phi_{\lambda}(X) ))$    \\
\hline
(viii) & $d \ell_{\gamma}(\mu)=\displaystyle \frac{2}{\pi} \mathrm{Re} <\Theta_\alpha, \mu>$  & $d \mathrm{Ext}_{\lambda}(\mu)= -2 \mathrm{Re} <\Phi_\lambda, \mu>$    \\
\hline
 (ix) & The horofuction boundary is  & The horofuction boundary  is    \\
& Thurston's boundary & Gardiner-Masur's boundary   \\
\hline
  (x) & $\mathcal{L}_X(\lambda)=\displaystyle\frac{\ell_X(\lambda)}{L_X}$ for $\lambda\in\mathcal{ML}$ & $\mathcal{E}_X(\lambda)=\displaystyle\frac{Ext^{1/2}_X(\lambda)}{K^{1/2}_X}$ for $\lambda\in\mathcal{ML}$   \\
\hline

\end{tabular}
\end{center}
\medskip
Most of the entries in this table are well know, some of them are known but need explanation, and some of the them are new and proved in this paper.
We now make a few comments on all the entries.
\medskip

(i) Stretch maps arise from the extremal problem of finding the best Lipschitz constant of maps homotopic to the identity between two hyperbolic structures on a surface $S$, in much the same way as Teichm\"uller maps arise from the extremal problem of finding the best quasiconformal constant of maps homotopic to the identity between two complex structures on $S$.

(ii) Stretch lines are geodesics for Thurston's metric. A stretch line is determined by a pair $(\mu,F)$ where $\mu$ is a complete lamination (not necessarily measured) and $F$ a measured foliation transverse to $\mu$.  Thurston proved that any two points in Teichm\"uller space can be joined by a geodesic which is a concatenation of stretch lines, but in general such a geodesic is not unique. Furthermore, there exist geodesics for Thurston's metric that are not concatenations of stretch lines. This contrasts with Teichm\"uller's theorem establishing that existence and uniqueness of geodesics joining any two distinct points, cf. \cite{T1939},
\cite{T1943} and Ahlfor's survery  \cite{Ahlfors-ICM1978}.

(iii) The left hand side is Thurston's definition of Thurston's metric, the infimum is over all homeomorphisms $f$ homotopic to the identity and $L(f)$ is the Lispschitz constant of such a homeomorphism. The right hand side is the definition of  the Teichm\"uller metric, the infimum is over all quasiconformal homeomorphisms $f$ homotopic to the identity and $K(f)$ is the dilatation of such a homeomorphism.

(iv) The left hand side is another expression (also due to Thurston) of Thurston's metric, and the right hand side is Kerckhoff's formula for the Teichm\"uller metric.

(v) Here, $V$ is a tangent vector to Teichm\"uller space at a point $X$. The left hand side formula is due to Thurston \cite{Thurston}; it is the infinitesimal form of Thurston's metric. The right hand side is an infinitesimal form of the  Teichm\"uller metric, and it is proved below (Theorem \ref{thm:teich}).

(vi) On the left hand side, $\mu$ is a complete geodesic lamination and the range of the map, $\mathcal{MF}(\mu)$, is the subspace of $\mathcal{MF}$ of equivalence classes of measured foliations that are transverse to $\mu$.
On the right hand side, $F$ is a measured foliation, and the range of the map, $\mathcal{MF}(F)$, is the subspace of $\mathcal{MF}$ consisting of equivalence classes of measured foliations that are transverse to $F$. The left hand side map is a homeomorphism defined in Thurston's paper \cite{Thurston}. The right hand side map is a homeomorphism that arises from the fact that a pair $F_1,F_2$ of measured foliations determines a unique point $X\in\mathcal{T}(S)$ and a unique quadratic differential $\Phi$ on $X$ such that $F_1$ ad $F_2$ are measure equivalent to the vertical and horizontal foliations of $\Phi$.
 Here, $F_h(\Phi_F(X))$ is the (equivalence class of the)  horizontal foliation of the quadratic differential on the Riemann surface $X$ having $F$ as vertical foliation.

(vii) The left hand side is an expression of the length of a complete lamination $\lambda$ as the geometric intersection with the horocyclic foliation associated to a completion $\lambda^*$ of $\lambda$. The formula is proved in \cite{P1} for the case where $\lambda$ is complete. The case where $\lambda$ is not complete follows easily from the geometric arguments used in that proof. The right hand side formula is due to  Kerckhoff, see Lemma \ref{lem:Ker} below.

(viii)   In the left hand side formula, $\mu$ is a Beltrami differential. The formula was given by Gardiner \cite{Gardiner1}, and in this formula $\alpha$ is (the homotopy class of) a simple closed curve and $\Theta_\alpha$  is the Poincar\'e series of $\alpha$, that is, a quadratic differential on the hyperbolic surface $X$ defined by $$\Theta_\alpha=\sum_{B\in <A>\setminus \Gamma}B^*(\frac{dz}{z})^2$$
where $\Gamma\subset \mathrm{PSL}(2,\mathbb{R})$ is a Fuchsian group associated to $X$ acting on the upper half-plane $\mathbb{H}^2$ and $A(z)=e^{\ell_\alpha}z$ is the deck transformation corresponding to the  simple closed geodesic $\alpha$.
The right hand side formula is called Gardiner's extremal length variational formula (see Lemma \ref{lem:Gardiner} below).

(ix) Walsh showed in \cite{Walsh} that the horofuction boundary of Thurston's metric is canonically identified with Thurston's boundary. Liu and Su showed in \cite{LS1} that the horofuction boundary of the Teichm\"uller metric is canonically identified with Gardiner-Masur's boundary.

(x) Here, we choose a basepoint  $X_0\in \mathcal{T}(S)$ and we fix a  complete geodesic lamination $\mu$ on $S$. For each $X \in \mathcal{T}(S)$, $L_X$  is the Lipschitz constant of the extremal Lipschitz map between $X_0$  and $X$ and $K_X$ is the quasiconformal dilatation of the extremal quasiconformal map between $X_0$ and $X$.

We illustrate the use of the functions $\mathcal{L}_X$ and $\mathcal{E}_X$. We use Thurston's homeomorphism $\phi_\mu:  \mathcal{T} (S) \to MF(\mu)$ that maps each $ X\in \mathcal{T}(S)$ to the horocylic foliation $F_\mu(X)$. The definition of this homeomorphism is recalled in \S \ref{s:Lipschitz} below. We denote the projective class of a measured foliation $F$ by $[F]$.
By a result in \cite{P1}, a sequence $(X_k)_{k\geq 1}$  in $ \mathcal{T}(S)$ coverges to a limit $[F]\in \mathcal{PMF}$ if and only if  $F_\mu(X_k)$ tends to infinity and $[F_\mu(X_k)]$ coverges to $[F]$.
We prove the following:
\begin{proposition}
If a sequence $(X_k)_{k\geq 1}$  in $ \mathcal{T}(S)$ coverges to a limit $[F]\in \mathcal{PMF}$,
then  $\mathcal{L}_{X_k}(\cdot)$ converges to $[F]$ in the following sense:
\medskip
\\
 $(\star)$    up to a subsequence, $\mathcal{L}_{X_k}(\cdot)$ converges to a positive multiple of $i(F,\cdot )$ uniformly on any compact subset of $ \mathcal{ML}$.
\end{proposition}
\begin{proof}
Assume that a sequence $(X_k)_{k\geq 1}$  coverges to $[F]\in \mathcal{PMF}$. Let ${L}_{X_k}=L_k$ and $F_\mu(X_k)=F_k$.
From Thurston's theory \cite{FLP}, there exists a sequence $(c_k)_{k\geq 1}$ of positive numbers such that  $c_k \to 0$ and
\begin{equation}\label{equ:con1}
c_k i(F_k, \lambda) \to i(F, \lambda)
\end{equation}
for each $\lambda\in\mathcal{ML} $.

By the Fundamental Lemma in \cite{P1}, there is  a uniform constant $C>0$ such that
\begin{equation}\label{equ:con2}
 i(F_k, \lambda) \leq  \ell_\lambda(X_k) \leq i(F_k, \lambda)+C.
\end{equation}
Note that $$c_k \ell_\lambda(X_k) \leq c_kL_k \ell_\lambda(X_0)$$
and for each $L_k$, there is a measured lamination $\lambda_k$ with $\ell_{\lambda_k}(X_0)=1$ such that $L_k= \ell_{\lambda_k}(X_k)$. As a result, we have
\begin{equation}\label{equ:con3}  \frac{c_k\ell_\lambda(X_k) }{\ell_\lambda(X_0)}  \leq c_kL_k\leq c_k \ell_{\lambda_k}(X_k).\end{equation}
From $(\ref{equ:con1})$, $(\ref{equ:con2})$, $(\ref{equ:con3})$, it follows that $c_kL_k$ is uniformly bounded from above and  uniformly bounded below away from zero. It follows that, up to  a subsequence,
$\frac{\ell_\lambda(X_k) }{L_k} $ converges to a positive multiple of $i(F, \lambda)$
for each $\lambda\in\mathcal{ML} $. Since pointed convergence of $\frac{\ell_\lambda(X_k) }{L_k} , \lambda\in\mathcal{ML} $ is equivalent to uniform convergence on compact subsets of $\mathcal{ML}$, property $(\star)$  holds.

\end{proof}

Concerning the right hand side formula in (x), a similar result was obtaind by Miyachi in \cite{Miyachi} :
\begin{proposition}
If a  sequence $(X_k)_{k\geq 1}$  in $ \mathcal{T}(S)$ coverges to a limit $P$ in Gardiner-Masur's boundary, then $\mathcal{E}_{X_k}(\cdot)$ converges to some function  $\mathcal{E}_P(\cdot)$ in the following sense: Up to a subsequence, $\mathcal{E}_{X_k}(\cdot)$ converges to a positive multiple of $\mathcal{E}_P(\cdot )$ uniformly on any compact subset of $ \mathcal{ML}$.

\end{proposition}

\section{Lipschitz Norm}\label{s:Lipschitz}

This section contains results of Thurston from his paper \cite{Thurston} that we will use later in this paper. We have provided proofs because at times Thurston's proofs in \cite{Thurston} are considered as sketchy.

 A geodesic lamination $\mu$ on a hyperbolic surface $X$ is said to be \emph{complete} if its complementary regions are all isometric to ideal triangles. (We note that we are dealing with laminations $\mu$ that are not necessarily measured, except if specified.) Associated with $(X, \mu)$ is a measured foliation $F_{\mu}(X)$, called the horocyclic foliation, satisfying the following three properties:
\begin{enumerate}[(i)]
\item $F_{\mu}(X)$ intersects $\mu$ transversely, and in each cusp of an ideal triangle in the complement of $\mu$, the leaves of the foliation are pieces of horocycles that make right angles with the boundary of the triangle;
\item  on the leaves of $\mu$, the transverse measure for $F_{\mu}(X)$ agrees with arclength;
\item there is a nonfoliated region at the centre of each ideal triangle of $S\setminus \mu$ whose boundary consists of three pieces of horocycles that are pairwise tangent (see Figure \ref{shrinkfig13bis}).
\begin{figure}[!hbp]
\centering
\psfrag{a}{\small \shortstack{horocycles\\perpendicular\\to the boundary}}
\psfrag{b}{\small \shortstack{horocycle of length 1}}
\psfrag{c}{\small \shortstack{non-foliated\\region}}
\includegraphics[width=0.5\linewidth]{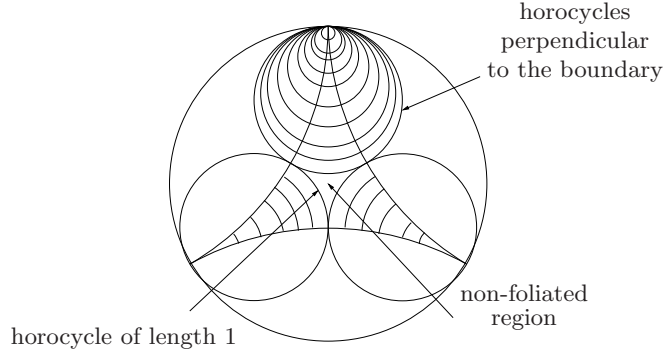}
\caption{\small {The horocyclic foliation of an ideal triangle.}}
\label{shrinkfig13bis}
\end{figure}

 \end{enumerate}
We denote by $\mathcal{MF}(\mu)$ the space of measured foliations that are transverse to $\mu$. Thurston \cite{Thurston} proved the following fundamental result.
\begin{theorem}
The map $\phi_\mu:  \mathcal{T} (S) \to \mathcal{MF}(\mu)$ defined by $ X \mapsto F_\mu(X)$ is a homeomorphism.
\end{theorem}

The  \emph{stretch line} directed by $ \mu$ and passing through $X\in  \mathcal{T} (S)$ is the curve
$$\mathbb{R} \ni t \mapsto X_t= \phi_\mu^{-1}(e^t F_\mu(X)).$$
We call a segment of a stretch line \emph{a stretch path}.

Suppose that $\mu$ is the support of a measured geodesic lamination $\lambda$.
Then, for any two points $X_s, X_t, s\leq t$ on the stretch line, their Lipschitz distance $d_L(X_s, X_t)$ is equal to $t-s$, and this distance is realized by
$$\log  \frac{\ell_{\lambda} (X_t)}{\ell_{\lambda} (X_s)}.$$

We denote by $\mathcal{ML}$ the space of measured geodesic laminations on $X$ and we let $\mathcal{ML}_1= \{\lambda\in \mathcal{ML}  \  | \  \ell_\lambda(X)=1 \}$.  We may identify $\mathcal{ML}_1$ with $\mathcal{PL}$, the space of projective measured laminations.

 Thurston \cite{Thurston} introduced a Finsler structure on $\mathcal{T}(S)$ by defining the Finsler norm of a tangent vector $V\in T_X\mathcal{T}(S)$  by the following formula :
\begin{equation}\label{eq:F}
\| V \|_L= \sup_{\lambda \in \mathcal{ML}} \frac{d \ell_{\lambda}(V)}{\ell_{\lambda}(X)} .
\end{equation}
Note that we may write
$$\| V \|_L=\sup_{\lambda \in \mathcal{ML}_1} d \ell_{\lambda}(V).$$

\begin{lemma}\label{lem:st}
Let $\Gamma_\mu(t)$  be a stretch line through $X$  with $\Gamma_\mu(0)=X$
and $\dot\Gamma_\mu(0)=V_0$.
For any  measured lamination  $\lambda$ supported by $\mu$ ($\lambda$ may be not unique), we have
$$\| V_0 \|_L= \frac{d \ell_{\lambda}(V_0)}{\ell_{\lambda}(X)} .$$

\end{lemma}
\begin{proof}
By multiplying the transverse measure by a constant, we may assume that $\ell_{\lambda}(X)=1$.  The stretch map from $X$ to $\Gamma_\lambda(t)$ is a Lipschitz map with Lipschitz constant $e^t$, and
the hyperbolic length of $\lambda$ in $\Gamma_\lambda(t)$ is equal to $e^t$.
As a result, for any measured lamination $\gamma\in \mathcal{ML}_1$,
$$\ell_{\gamma}(\Gamma_\lambda(t))\leq e^t.$$
It follows that
$$\sup_{\gamma \in \mathcal{ML}_1} {d \ell_{\gamma}(V_0)} = d \ell_{\lambda_0}(V_0).$$
\end{proof}

\begin{theorem}[Thurston \cite{Thurston} p. 20]\label{thm:lip}
$\| \cdot \|_L $ is the infinitesimal norm of Thurston's asymmetric distance
 $d_L$.
\end{theorem}
\begin{proof}

Recall that the Finsler norm $\| \cdot \|_L$ induces an (asymmetric) distance:
$$d(X,Y)= \inf_{\Gamma} \int_0^1 \| \dot\Gamma \|_L,$$
where the infimum is taken over all piecewise $C^1$ curves $\Gamma:[0,1]\to \mathcal{T}$ in Teichm\"uller space joining $X$ to $Y$.

By compactness of $\mathcal{PML}$, there exists an element $\lambda\in\mathcal{ML}$ such that the distance $d_L(X,Y)$ is attained at $\lambda$, that is,
$$d_L(X,Y)=  \log  \frac{\ell_{\lambda}(Y)}{\ell_{\lambda}(X)} .$$

For any piecewise $C^1$ path $\Gamma: [0, 1] \to \mathcal{T} (S)$ satisfying $\Gamma(0)=X$ and  $\Gamma(1)=Y$, we have
\begin{eqnarray*}
d_L(X,Y) &=&  \log  \frac{\ell_{\lambda}(Y)}{\ell_{\lambda}(X)}  \\
&=& \int_0^1   d \log\ell_{\lambda} ( \dot\Gamma(t) )  \\
&\leq & \int_0^1   \| \dot\Gamma(t) \|_L.
\end{eqnarray*}
Therefore, $d_L(X,Y)\leq d(X,Y). $

To prove the reverse inequality, we use a result proved by Thurston in \cite{Thurston}, namely, that any two points in Teichm\"uller space
can be connected by a finite concatenation of stretch paths, each of which stretches along some common measured geodesic lamination $\lambda$.

As a result, for any distinct $X$ and $Y$ in $\mathcal{T}(S)$, we can assume that there exists a path $\Gamma$ connecting $X$ and $Y$,
with  $\Gamma$ being a concatenation $\Gamma_1 * \cdots * \Gamma_n$ of stretch paths. Up to reparametrization, we may assume that each $\Gamma_i$
is defined on $[0, 1]$. Furthermore, by a result of Thurston, we may assume that each stretch path $\Gamma_i$ is directed by a complete geodesic lamination  $\mu_i$ that contains a common measured lamination $\lambda$, which is the \emph{maximally stretched lamintation} from $X$ to $Y$. This is a consequence of Theorem 8.2 in Thurston \cite{Thurston}, in which Thurston shows  that there is a unique maximal ratio-maximizing  chain recurrent lamination which contains all other ratio-maximizing chain recurrent lamitations.
For our purposes, we take $\lambda$ to be the maximal (with respect to inclusion) measured lamination contained in $\mu(X,Y)$.

 Along each $\Gamma_i$, by Lemma \ref{lem:st}, we have
$$\sup_{\nu \in \mathcal{ML}} \frac{d \ell_{\nu}(\dot\Gamma_i)}{\ell_{\nu}} =  \frac{d \ell_{\lambda}(\dot\Gamma_i)}{\ell_{\lambda}} .$$

Now we have
\begin{eqnarray*}
 \int_\Gamma   \| \dot\Gamma \|_L
&=& \sum_{i=1}^n  \int_{\Gamma_i}   \| \dot\Gamma \|_L  \\
&=& \sum_{i=1}^n  \int_{\Gamma_i}  \frac{d \ell_{\lambda}(\dot\Gamma_i)}{\ell_{\lambda}}  \\
&=& \sum_{i=1}^n   \log  \frac{\ell_{\lambda}(\Gamma_i(1))}{\ell_{\lambda}(\Gamma_i(0))}  \\
&=&  \log  \frac{\ell_{\lambda}(Y)}{\ell_{\lambda}(X)} .
\end{eqnarray*}
For the last equality, note that we have $\Gamma_i(1)=\Gamma_{i+1}(0)$ for each $i$ and $\Gamma_1(0)=X, \Gamma_n(1)=Y$. It follows that $d(X,Y) \leq d_L(X,Y)$.

\end{proof}

\section{ Teichm\"uller Norm}
Extremal length is an important tool in the study of the Teichm\"uller metric. The notion is due to Ahlfors and Beurling \cite{AB1}, see also \cite{AB}. We briefly recall the definition. Given a Riemann surface $X$, a \emph{conformal metric} $\sigma$ on $X$ is a metric that is locally of the form
$\sigma(z)|dz|$ where $z$ is a local holomorphic parameter and  $\sigma(z)\geq 0$ is a Borel measurable function in the local chart.  We define the
$\sigma$-area of $X$ by
$$ A(\sigma)=\int_X \sigma^2(z)|dz|^2 .$$
Given a homotopy class of simple closed curves $\alpha$, its $\sigma$-length is defined by
$$L_\sigma(\alpha)=\inf_{\alpha' }\int_{\alpha'}\sigma(z)|dz|,$$
where the infimum is taken over all essential simple closed curves $\alpha' $ in the homotopy class  $\alpha$. (Note that it follows from the invariance property of the expression $\sigma(z)|dz|$ that the two integrals $A(\sigma)$ and $L_\sigma(\alpha)$ are well defined, that is, they can be computed in the holomorphic local coordinates and the value obtained does not depend on the choice of the coordinates, see \cite{AB}.)

With the above notation, we define the extremal length of $\alpha$ on $X$  by
$$\mathrm{Ext}_\alpha(X)=\sup_\sigma \frac{L_\sigma^2(\alpha)}{A(\sigma)},$$
where  $\sigma(z)|dz|$ ranges over all conformal metrics on $X$  satisfying
$0<A(\sigma)<\infty$.

By a result of Kerckhoff \cite{Kerckhoff}, there is a unique
continuous extension of the extremal length function  to the space of measured foliations $\mathcal{MF}$,
with $\mathrm{Ext}_{a\gamma}(X)=a^2\mathrm{Ext}_{\gamma}(X)$ for  $a>0$ and $\gamma$  a homotopy class of simple closed curves on $X$.
Recall that a tangent vector to Teichm\"uller space at a point $X$ can be represented by a Beltrami differential
$\mu=\mu(z)\frac{d  \bar z}{d z}$ \cite{IT}. (Note that in this section we use the letter $\mu$ to denote a Beltrami differential since this is the traditional notation, although in other sections the same letter is used to denote a lamination. Hopefully, there will be no confusion.)
Using the extremal length function, we define a norm on the tangent space at $X$ by setting:
\begin{equation}\label{eq:E}
\| \mu \|_E= \sup_{\lambda \in \mathcal{MF}} \frac{d  \mathrm{Ext}^{1/2}_{\lambda}(\mu)}{ \mathrm{Ext}^{1/2}_{\lambda}(X)} .
\end{equation}

The following theorem is due to Hubbard-Masur \cite{HM}; we refer to Kerckhoff \cite{Kerckhoff} for a short proof.
\begin{theorem} \label{th:HM}
For any Riemann surface $X$ and for any measured foliation $\lambda$ on $X$, there is exactly one quadratic differential, denoted by $\Phi_\lambda$, whose vertical measured foliation is measure-equivalent to $\lambda$.

\end{theorem}
\begin{lemma}[Gardiner \cite{Gardiner}] \label{lem:Gardiner}
The extremal length function $\mathrm{Ext}_{\lambda}$ is differentiable and we have the following formula, called the ``first variational formula":
\begin{equation}\label{equ:Gar}
d \mathrm{Ext}_{\lambda}(\mu)= -2 \mathrm{Re} <\Phi_\lambda, \mu>,
\end{equation}
where  $<\Phi_\lambda, \mu>$ is the natural pairing
$$<\Phi_\lambda, \mu>=\iint_X  \Phi_\lambda(z)\mu(z) d x d y.$$
\end{lemma}
The following observation is due to Kerckhoff \cite{Kerckhoff}:
\begin{lemma}\label{lem:Ker}
With the above notation,
\begin{equation}\label{equ:Ker}
 \mathrm{Ext}_{\lambda}(X)= \iint_X  |\Phi_\lambda(z)| d x d y.
\end{equation}
\end{lemma}
\begin{proof}
The proof we give here is due to Ivanov \cite{Ivanov}. We include it for completeness.

By continuity and the density of weighted simple closed curve in $\mathcal{MF}$, it suffices to prove $(\ref{equ:Ker})$ for the case where $\mu=a\gamma \in \mathcal{MF}$, where $\gamma$ is (the homotopy class of) a simple closed curve and $a >0$.

Let $\Phi$ be the one-cylinder Strebel differential on $X$ determined
by  $a\gamma$. The complement of the vertical critical leaves of $\Phi$ is a cylinder foliated by circles isotopic to $\gamma$. Let us also set $\rho=|\Phi|^{1/2}|dz|$. Then $\rho$ is a flat metric on $S$, with a finite number of singular points, which are conical singularities. Measured in the flat metric $\rho$, the circumference and height of the cylinder are equal to $L_\rho(\gamma)$
and $a$ respectively.
By a theorem of Jenkins-Strebel \cite{Strebel}, the extremal length $\mathrm{Ext}_{X}(\gamma)$ of $\gamma$ is equal to
$$\mathrm{Ext}_{\gamma}(X)=\frac{L_\rho(\gamma)}{a},$$
where $\rho=|\Phi|^{1/2}|dz|$.

 The area $A(\rho)=\iint_X  |\Phi_\lambda(z)| d x d y$ of the cylinder is equal to $aL_\rho(\gamma)$. As a result,
$$\mathrm{Ext}_{a\gamma}(X)=a^2\mathrm{Ext}_{\gamma}(X)=a {L_\rho(\gamma)}=A(\rho).$$

\end{proof}

The next result follows from $(\ref{equ:Gar})$  and $(\ref{equ:Ker})$.

\begin{proposition}\label{pro:equality}
The norm $\| \mu \|_E$ satisfies
\begin{equation}\label{equ:norm}
\| \mu \|_E= \sup_{\lambda \in \mathcal{MF}} \frac{-\mathrm{Re} <\Phi_\lambda, \mu>}{\iint_X  |\Phi_\lambda(z)| d x d y }= \sup_{\| \Phi \|=1} {\mathrm{Re} <\Phi, \mu>},
\end{equation}
where $\Phi$ varies over all holomorphic quadratic differentials $\Phi$.
\end{proposition}

Note that we have  $\| \Phi \|=\iint_X  |\Phi (z)| d x d y$.

\bigskip

Now we consider the  Teichm\"uller metric on $\mathcal{T} (S)$.
 We recall that it is defined by
$$d_T(X,Y):=\frac{1}{2} \inf_f \log K(f)$$
where $f:X \to Y$ is a
quasi-conformal map homotopic to the identity map of $S$ and
$$K(f)=\sup_{x\in X}K_x(f)\geq 1$$
 is the quasi-conformal dilatation of $f$ (the $\sup$ here denotes essential supremum), with
 $$K_x(f)=\frac{|f_z(x)|+|f_{\bar z}(x)|}{|f_z(x)|-|f_{\bar z}(x)|}$$
 being the pointwise quasiconformal dilatation at the point $x\in X$ with local conformal
coordinate $z$.

Teichm\"uller's theorem states that  given any $X,Y \in \mathcal{T}(S)$, there exists a unique
quasi-conformal map $f: X \to Y$, called the Teichm\"uller map, such that
$$d_T(X,Y)=\frac{1}{2} \log K(f).$$
The Beltrami coefficient $\mu:=\displaystyle  \frac{\bar\partial f}{\partial f}$ is of the form $\mu=\displaystyle  k\frac{\bar \Phi}{|\Phi|}$
for some quadratic differential $\Phi$ on $X$ and some constant $k$ with $0\leq k <1$.
In some natural coordinates given by $\Phi$ on $X$ and for some associated quadratic differential $\Phi'$ on $Y$, the Teichm\"uller map
$f$ is given by $f(x+iy)=K^{1/2}x+ i K^{-1/2}y$, where $K=K(f)=\displaystyle  \frac{1+k}{1-k}$.

It is known that the Teichm\"uller metric is a Finsler metric and that between any two points in $\mathcal{T}(S)$ there is exactly one geodesic.
A geodesic ray with initial point $X$ is given by the one-parameter family of Riemann surfaces $\{X_t\}_{t\geq 0}$,
where there is a holomorphic quadratic differential $\Phi$ on $X$ and a family of
Teichm\"uller maps $f_t:X\to X_t$, with initial Beltrami differential $\mu(f_t)=\displaystyle  \frac{e^{2t}-1}{e^{2t}+1}\frac{\bar \Phi}{|\Phi|}$. Here  $\mu(f_t)$ is chosen such that the geodesic ray has unit speed, that is, $d_T(X_s,X_t)=t-s$ for all $s\leq t$.
We also recall the following formula due to Kerckhoff \cite{Kerckhoff}.
\begin{theorem}
Let $X, Y$ be any two points in $\mathcal {T}(S)$. Then
$$d_T(X,Y)=\frac{1}{2} \log \sup_{\lambda} {\frac{\mathrm{Ext}_{\lambda}(Y)}{\mathrm{Ext}_{\lambda}(X)}},$$
where $\lambda$ ranges over elements in $\mathcal{MF}$.
\end{theorem}

Now we can prove the following:
\begin{theorem}\label{thm:teich}
 The metric induced by the norm $\| \cdot \|_E$ defined in (\ref{eq:E}) is the  Teichm\"uller metric.
\end{theorem}
\begin{proof}
Denote by $d_E$ the length metric on $\mathcal {T}(S)$   induced by the norm $\| \mu \|_E$.
For any $X, Y$  in $\mathcal {T}(S)$, by Kerckhoff's formula, the Teichm\"uller distance is realized by
$$d_T(X,Y)=\frac{1}{2} \log {\frac{\mathrm{Ext}_{\lambda}(Y)}{\mathrm{Ext}_{\lambda}(X)}}$$
for some measured foliation $\lambda$. An argument similar to the one in the first part of the proof of Theorem \ref{thm:lip} shows that $d_T(X,Y)\leq d(X,Y)$.

For the converse, let $\Gamma(t)$ be a Teichm\"uller geodesic connecting $X$ and $Y$. We parametrize $\Gamma$ with unit speed, and define it on the interval $[0, T]$ with $T=d_T(X,Y)$ with $\Gamma(0)=X$ and $\Gamma(T)=Y$.

For each $t$ in $[0,T]$, the Teichm\"uller map between $X=\Gamma(0)$ and $\Gamma(t)$ has quasiconformal dilatation $e^{2t}$. It follows from the geometric definition of a quasiconformal map that for any $\lambda$ in $\mathcal{MF}$, we have
$$\mathrm{Ext}_{\lambda}(\Gamma(t)) \leq e^{2t} \mathrm{Ext}_{\lambda}(\Gamma(0)).$$
Taking square roots, we get
$$\mathrm{Ext}^{1/2}_{\lambda}(\Gamma(t)) \leq e^{t} \mathrm{Ext}^{1/2}_{\lambda}(\Gamma(0))$$
and then
\begin{eqnarray*} \lim_{t\to 0^+} \frac{\mathrm{Ext}^{1/2}_{\lambda}(\Gamma(t))-\mathrm{Ext}^{1/2}_{\lambda}(\Gamma(0))}{t \mathrm{Ext}^{1/2}_{\lambda}(\Gamma(t))}&\leq &
\lim_{t\to 0^+} \frac{e^t\mathrm{Ext}^{1/2}_{\lambda}(\Gamma(0))-\mathrm{Ext}^{1/2}_{\lambda}(\Gamma(0))}{t \mathrm{Ext}^{1/2}_{\lambda}(\Gamma(t))}\\ &=&
\lim_{t\to 0^+} \frac{e^t-1}{t}\\ &=& 1.
\end{eqnarray*}
As a result, $d \log  \mathrm{Ext}^{1/2}_{\lambda} (\dot\Gamma(0)) \leq 1$ and then
$$d \log  \mathrm{Ext}^{1/2}_{\lambda} (\dot\Gamma(t)) \leq 1$$ for all $t\in\mathbb{R}$.

Considering the integral of the norm $\| \cdot \|_E$ along $\Gamma(t)$, we have
$$\int_0^T \| \dot\Gamma(t) \|_E \leq T=d_T(X,Y).$$
This proves that $d(X,Y) \leq d_T(X,Y)$.
\end{proof}

Combining Proposition \ref{pro:equality} and Theorem \ref{thm:teich} , we have the following corollary.
\begin{corollary}\label{coro:teich}
The Teichm\"uller Finsler norm, denoted by $\| \mu \|_T$,  is given by
$$\| \mu \|_T=\sup_{\| \Phi \|=1} \mathrm{Re} <\Phi, \mu>,$$
where $\Phi$ varies over all holomorphic quadratic differentials $\Phi$.

\end{corollary}
The above result was already known and can be obtained by using the famous Reich-Strebel inequality \cite{RS}. The result means that the Teichm\"uller  norm $\| \cdot \|_T$ is dual to the $L^1$-norm
$$\| \Phi \|=\iint_X  |\Phi (z)| d x d y$$
on the cotangent space.

\section{Convex embedding of measure foliation space using extremal length}

 Following a usual trend, we shall call the boundary of a convex body a \emph{convex sphere}. Note that the boundary of a convex body is always homeomorphic to a sphere.

 Thurston \cite{Thurston} proved that for any $X\in \mathcal{T} (S)$, the function $d \ell: \mathcal{ML}_1 \to\mathrm{T^*}_X \mathcal{T} (S)$
 defined by $ \lambda\mapsto  d \ell_{\lambda}$  embeds $\mathcal{ML}_1$ as a convex sphere in the cotangent space of Teichm\"uller space at $X$ containing the origin. This embedding is the dual of the boundary of the unit ball in the tangent space representing vectors of norm one for the Finsler structure associated to Thurston's metric. We prove an analogous result for the extremal length function.

\begin{theorem}
Given a point $X$ in $\mathcal{T} (S)$, let $\mathcal{PF}=\mathcal{MF}_1= \{\lambda\in \mathcal{MF}  \  | \  \mathrm{Ext}_\lambda(X)=1 \}$. Then, the function $\mathcal{PF}  \to \mathrm{T^*}_X \mathcal{T} (S)$ defined by
$\lambda\mapsto d \mathrm{Ext}_{\lambda}$  embeds $\mathcal{PF}$ as a convex sphere in $\mathrm{T^*}_X \mathcal{T} (S)$ containing the origin. This sphere is the boundary of the  dual of the unit ball in the tangent space representing vectors of norm one for the Finsler structure associated to the Teichm\"uller metric.
\end{theorem}
\begin{proof} We fix a point $X$ in $\mathcal{T}(S)$.
By Gardiner's formula,
$$d \mathrm{Ext}_{\lambda}(\mu)= -2 \mathrm{Re} <\Phi_\lambda, \mu>$$
(see Theorem \ref{th:HM} for the notation).
Kerckhoff proved  that if a sequence  $\lambda_n$ in $ \mathcal{PF}$ converges to $\lambda$  then $\|\Phi_{\lambda_n}-\Phi_\lambda \| \to 0$ (see Page 34-35 of  Kerckhoff's paper \cite{Kerckhoff}, where he proved the result for the case where each $\lambda_n$ is a simple closed curve. Since the subset of weighted simple closed curves is dense in $\mathcal{MF}$, the result we need follows). It follows that the map $d\mathrm{Ext}$ is continuous.

Moreover, since the pairing between quadratic differentials and harmonic Beltrami differentials is nondegenerate \cite{IT}, if $d \mathrm{Ext}_{\lambda}=d \mathrm{Ext}_{\lambda'}$, then $\Phi_{\lambda}=\Phi_{\lambda'}$.
This means that $\lambda=\lambda'$. Therefore, the map $d\mathrm{Ext}$ is injective.

Let $d\mathrm{Ext} (\mathcal{PF})$ be the image of $\mathcal{PF}$ under the map $d \mathrm{Ext}$, that is,
$$d\mathrm{Ext} (\mathcal{PF})=\{d\mathrm{Ext}_{\lambda}\in \mathrm{T^*}_X \mathcal{T} (S) \ | \ \lambda\in \mathcal{PF}\}.$$
Since $\mathcal{PF}$ is homeomorphic to a sphere, by invariance of domain, $d\mathrm{Ext} (\mathcal{PF})$   is open and $d\mathrm{Ext}$ is a homeomorphism between $\mathcal{PF}$ and $d\mathrm{Ext} (\mathcal{PF})$.

For any $\lambda\in \mathcal{PF}$, consider the Beltrami differential $$V_\lambda=-\frac{\overline{\Phi_{\lambda}}}{2|\Phi_\lambda |}.$$
It is a tangent vector in $\mathrm{T}_X \mathcal{T} (S)$. By Gardiner's formula,
\begin{eqnarray*}
d \mathrm{Ext}_\lambda(V_\lambda)&=&-2 \mathrm{Re}<\Phi_\lambda, -\frac{\overline{\Phi_{\lambda}}}{2|\Phi_\lambda |}> \\
&=&\iint_X |\Phi_\lambda| dx dy=\mathrm{Ext}_\lambda(X) =1.
\end{eqnarray*}
As a result, the derivative  $d \mathrm{Ext}_{\lambda}$ in the direction $V_\lambda$  is $1$. Using H\"older's inequality, 
for any $\gamma\in \mathcal{PF}$, we have
\begin{eqnarray*}
d \mathrm{Ext}_\gamma(V_\lambda)&=& \iint \Phi_\gamma \frac{\overline{\Phi_{\lambda}}}{|\Phi_\lambda |} dx dy\\
&\leq &\iint_X |\Phi_\gamma|dx dy =\mathrm{Ext}_\gamma(X)=1.
\end{eqnarray*}
Moreover, equality holds if and only if $\gamma=\lambda$.
Therefore, the derivative $d \mathrm{Ext}_{\gamma}(V_\lambda)$ of any other $\gamma\in  \mathcal{PF}$ is strictly less than $1$. Denote by $C(d\mathrm{Ext} (\mathcal{PF}))$ the convex hull of $d\mathrm{Ext} (\mathcal{PF})$ in $\mathrm{T^*}_X \mathcal{T} (S)$. Thus, $V_\lambda$ defines a non-constant linear functional on the cotangent space which attains its maximal value on $C(d\mathrm{Ext} (\mathcal{PF}))$ at $d \mathrm{Ext}_{\lambda}$. As a result, $d \mathrm{Ext}_{\lambda}$ is an extreme point of  the convex set $C(d\mathrm{Ext} (\mathcal{PF}))$.

We now use Thurston's argument in his proof of Theorem 5.1 of \cite{Thurston}. The set of extreme points of any convex set is a sphere of some dimension. Since the dimension of $d\mathrm{Ext} (\mathcal{PF})$ is one less than the dimension of $\mathrm{T^*}_X \mathcal{T} (S)$, the set of extremal points of $C(d\mathrm{Ext} (\mathcal{PF}))$ coincides with $d\mathrm{Ext} (\mathcal{PF})$. As a  result, $d\mathrm{Ext} (\mathcal{PF})$ is a convex sphere.
To see that the convex set $C(d\mathrm{Ext} (\mathcal{PF}))$ contains the origin in its interior, note that since this convex hull has a nonempty interior, there must be at least one line through the origin of $\mathrm{T^*}_X \mathcal{T} (S)$ which intersects $d\mathrm{Ext} (\mathcal{PF})$ in at least two points.
For each of these points, there is a linear functional which attains its positive maximum value there. It follows that $0$ must separate the two points, so $0$ is a convex combination of them.
\end{proof}

\section{Comparison between the Lipschitz and Teichm\"uller norms}
From an inequality called Wolpert's inequality \cite{Wolpert}, we have $d_L \leq 2d_T$. (In fact, the inequality is contained in Sorvali's paper \cite{Sor}, and it was rediscovered by Wolpert. ) Choi and Rafi \cite{CR} proved that the two metrics are quasi-isometric in any thick part of the Teichm\"uller space, but that there are sequences $(X_n), (Y_n)$ in the thin part, with $d_L(X_n,Y_n)\to 0$.

It is interesting to compare the Lipschitz norm and the Teichm\"uller norm. In particular, we ask the following:
\begin{question} Determine a function $C(\epsilon)$ such that
$$\| V \|_L \leq \| V \|_T \leq C(\epsilon) \| V \|_L $$
for any $V\in T_X\mathcal{T} (S) $ and for all $X\in \mathcal{T}_\epsilon (S)$, the $\epsilon$-thick part of  $\mathcal{T} (S)$.
\end{question}

Since $ \| V \|_L $ (respectively $ \| V \|_T$) is given by the logarithmic derivative of the hyperbolic (respectively extremal) length function of measured laminations, a further study of the relation between hyperbolic and extremal length may give the answer to the above question.

\begin{question} Consider the length-spectrum metric $d_{ls}$ on $\mathcal{T}(S)$. This can be defined as a symmetrization of Thurston's metric\footnote{Note that historically the length spectrum metric was not introduced as a symmetrization. It was  first defined by Sorvali \cite{Sorvali}, before Thurston introduced his  asymmetric metric.}, by the formula
\begin{eqnarray*}
d_{ls}(X,Y)&=& \max\{\log \sup_{\gamma}  \frac{\ell_Y(\gamma)}{\ell_X(\gamma)}, \log \sup_{\gamma}  \frac{\ell_X(\gamma)}{\ell_Y(\gamma)}\}  \\
&=&\max\{d_L(X,Y), d_L(Y,X)\}.  \\
\end{eqnarray*}

Is the length-spectrum metric $d_{ls}$ a Finsler metric ? If yes, give a formula for the infinitesimal norm of a vector on Teichm\"uller space with respect to this Finsler structure.
\end{question}
There are comparisons between the length-spectrum metric and the Teichm\"uller metric. It was shown by Liu and Su \cite{LS} that the divergence of $ d_{ls}$ and $d_T$ is only caused by the action of the mapping class group. In fact, they showed that $d_{ls}$ and $d_T$ are ``almost'' isometric on the moduli space and that the two metrics on the moduli space determine the same  asymptotic  cone.

\begin{question} One can also seek for  results analogous to those presented here for the weak metric  $L^ *$ dual to Thurston's metric, that is, the weak metric on Teichm\"uller space defined by
\[L^*(X,Y)=L(Y,X),\]
as well as for the length spectrum metric $d_{ls}$ on the same space.
\end{question}

 \end{document}